\documentclass[12pt]{amsart}

\usepackage{amsfonts}
\usepackage{amsmath}
\usepackage{amsrefs}
\usepackage{tikz-cd}
\usepackage[all]{xy}
\usepackage{amssymb}
\usepackage{amsthm}
\usepackage[english]{babel}
\usepackage{color}
\usepackage{epsfig}
\usepackage{graphicx}
\usepackage{hyperref}
\usepackage[utf8]{inputenc}
\usepackage{mathrsfs}
\usepackage{xcolor}
\usepackage{amsfonts}
\usepackage{amsmath}
\usepackage{setspace} % Paquete para controlar el espaciado entre líneas
\usepackage{amsrefs}
\usepackage{amssymb}
\usepackage{amsthm}
\newtheorem*{theorem*}{Theorem}
\usepackage[english]{babel}
\usepackage{color}
\usepackage{epsfig}
\usepackage{graphicx}
\usepackage{hyperref}
\usepackage[utf8]{inputenc}
\usepackage{mathrsfs}
\usepackage{nicefrac}
\usepackage{psfrag}
\usepackage{xcolor}
\usepackage{comment}
\usepackage{tikz}
\usepackage{fancyhdr}
\usepackage{lipsum}  
\usepackage{pgfplots}
\usetikzlibrary{arrows.meta, bending, positioning, decorations.markings,decorations.pathreplacing, patterns,snakes}
\hypersetup{
 colorlinks,
 citecolor=blue,
 filecolor=blue,
 linkcolor=blue,
 urlcolor=blue}

\newcommand{\R}{\mathbb R}
\newcommand{\Z}{\mathbb Z}

\newcommand{\U}{\mathcal U}

\newcommand{\abs}[1]{\left| #1 \right|}

\theoremstyle{plain}
\newtheorem{theorem}{Theorem}[section]

\newtheorem{corollary}[theorem]{Corollary}
\newtheorem{lemma}[theorem]{Lemma}
\newtheorem{proposition}[theorem]{Proposition}

\theoremstyle{definition}
\newtheorem{definition}[theorem]{Definition}
\newtheorem{remark}[theorem]{Remark}
\newtheorem{example}[theorem]{Example}

\theoremstyle{remark}

\newcommand{\CBstable}{\ensuremath{\mathrm{CB\text{-}stable}}\ } 
\usepackage{setspace}
\usepackage[margin=2cm]{geometry}
\definecolor{ao}{rgb}{0.0, 0.5, 0.0}

\providecommand{\abs}[1]{\lvert#1\rvert}

\makeatletter
\newsavebox\myboxA
\newsavebox\myboxB
\newlength\mylenA
\newcommand*\xoverline[2][0.75]{%
    \sbox{\myboxA}{$\m@th#2$}%
    \setbox\myboxB\null% Phantom box
    \ht\myboxB=\ht\myboxA%
    \dp\myboxB=\dp\myboxA%
    \wd\myboxB=#1\wd\myboxA% Scale phantom
    \sbox\myboxB{$\m@th\overline{\copy\myboxB}$}%  Overlined phantom
    \setlength\mylenA{\the\wd\myboxA}%   calc width diff
    \addtolength\mylenA{-\the\wd\myboxB}%
    \ifdim\wd\myboxB<\wd\myboxA%
       \rlap{\hskip 0.5\mylenA\usebox\myboxB}{\usebox\myboxA}%
    \else
        \hskip -0.5\mylenA\rlap{\usebox\myboxA}{\hskip 0.5\mylenA\usebox\myboxB}%
    \fi}
\makeatother

\begin{document}

\author[A. Artigue]{Alfonso Artigue}
\address{Departamento de Matemática y Estadística del Litoral\\
Centro Universitario Regional Litoral Norte\\
Universidad de la Rep\'ublica\\
Florida 1065, Paysandú, Uruguay.\\}
\email{aartigue@litoralnorte.udelar.edu.uy}

\author[L. Ferrari]{Luis Ferrari}
\address{Departamento de Matem\'atica y Aplicaciones\\
Universidad de la Rep\'ublica\\
Maldonado, Tacuaremb\'o entre Av. Artigas y Aparicio Saravia, CP 20000, Uruguay}
\email{luisrferrari@gmail.com}

\title{Metric-Independent Expansiveness}
\date{\today}

\begin{abstract}
In this article we introduce and study a natural form of expansivity, that we call \textit{metric-independent expansiveness}, for group actions on metrizable spaces. This notion means \textit{expansive with respect to every compatible metric}. For actions on locally compact $\sigma$-compact metric spaces, we show that this property admits a purely topological characterization: it is equivalent to what we call \textit{cocompact expansivity} and to the existence of an expansive extension to the one-point compactification.
We apply this characterization to ordinal spaces and to totally bounded spaces, obtaining criteria and examples that distinguish expansive actions from genuinely metric-independent ones.
A central theme in these applications is that metric-independent expansiveness can be recovered from expansive compact dynamics when the boundary of the compactification is dynamically isolated. Finally, we introduce the notion of Cauchy expansiveness and prove that every uniformly continuous Cauchy expansive action extends uniquely to an expansive action on the completion.
\end{abstract}

\maketitle

%\tableofcontents

\section{Introduction}

In the study of dynamical systems, conjugacy is usually taken as a natural notion of equivalence. To fix ideas, consider two maps $f\colon X\to X$ and $g\colon Y\to Y$ defined on topological spaces $X$ and $Y$. We say that $f$ and $g$ are \textit{conjugate} if there exists a homeomorphism $h\colon X\to Y$ such that $h\circ f=g\circ h$. From this viewpoint, the relevant dynamical properties and invariant sets should be preserved under conjugacy. For example, transitivity and topological mixing are invariant dynamical properties, and the non-wandering set and the set of periodic points are also preserved by a conjugacy.
However, some important properties are invariant only under additional assumptions. For instance, expansiveness is well known to be invariant under conjugacy in the compact metrizable setting, but there are expansive homeomorphisms that lose this property under conjugacies when compactness is dropped. Recall that a homeomorphism $f$ of a metric space $(X,d)$ is \textit{expansive} if there exists $c>0$ such that for every pair of distinct points $x,y\in X$, one has $d(f^i(x),f^i(y))>c$ for some $i\in\Z$.
For example, the map $f\colon\R\to\R$ given by $f(x)=2x$ is expansive for the usual distance on $\R$, but it is easy to change the metric and destroy expansiveness (see Example~\ref{ejemplo2}). An even simpler example is the identity map on the discrete space $\Z$, whose expansiveness clearly depends on the metric.

This phenomenon suggests that classical metric expansiveness is not, in general, a purely topological property, at least outside the compact metrizable setting. Nevertheless, there are topological formulations of expansiveness that do not depend on the choice of a particular metric.
For homeomorphisms of compact Hausdorff spaces
% , Achigar, Artigue and Monteverde
in \cite{Achigar} it was introduced orbit expansiveness by means of finite open covers.
Later, an analogous cover-theoretic approach for group actions on compact Hausdorff spaces was developed in \cite{Ferrari2025}, and this is the formulation that will be relevant for us.
Let us briefly recall this notion. If $\mathcal U$ is a finite open cover of $X$ and $A\subseteq X$, we write $A\prec\mathcal U$ whenever there exists $U\in\mathcal U$ such that $A\subseteq U$. An action $\Phi\colon G\times X\to X$ is said to be \textit{expansive by coverings} if there exists a finite open cover $\mathcal U$ of $X$ such that, for every $x,y\in X$,
\[
\{\Phi_g(x),\Phi_g(y)\}\prec\mathcal U \text{ for every } g\in G \quad\Longrightarrow\quad x=y.
\]
In that case, $\mathcal U$ is called an \textit{expansivity cover}. In the compact metrizable case, this cover-theoretic formulation agrees with the usual metric definition of expansiveness; see \cite{Achigar,Ferrari2025}. Thus, although metric expansiveness may fail to be preserved when one changes the compatible metric on a non-compact metrizable space, the cover approach shows that there is a genuinely topological core behind the concept.
A natural way to isolate this idea in the non-compact setting is to encode the behavior at infinity through compactifications. Among them, the one-point compactification is especially well suited for locally compact spaces, since it records whether different orbits may remain dynamically indistinguishable while escaping every compact set. From this perspective, expansiveness should not only separate orbits inside the space, but should also prevent distinct points from becoming asymptotically invisible near infinity. This idea leads naturally to the notions of cocompact expansivity and expansive extendibility to the one-point compactification, which will play a central role throughout the paper.
Our main goal is to show, in the locally compact $\sigma$-compact setting, that these notions are equivalent to metric-independent expansiveness.
% \rojo{Roughly speaking, our main result shows that, for actions on metrizable locally compact $\sigma$-compact spaces, metric-independent expansiveness can be characterized purely in topological terms. More precisely, in that setting we prove that metric-independent expansiveness is equivalent to cocompact expansivity, and also equivalent to the existence of an expansive extension to the one-point compactification.
Thus, the problem of deciding whether expansiveness depends on the choice of a compatible metric can be completely reformulated in terms of the topology at infinity.

We then apply this characterization to several classes of spaces. On the one hand, ordinal spaces provide a natural source of scattered, non-compact metrizable spaces where the interaction between topology and dynamics can be analyzed rather explicitly. On the other hand, totally bounded spaces show that metric-independent expansiveness is strongly influenced by the way the action sits inside a compact ambient space. In this direction, a key role is played by the dynamical behavior of the boundary: when the remainder of a compactification is dynamically isolated, expansiveness on the compact space can be transferred to the invariant open part. This shows that metric-independent expansiveness is substantially more rigid than ordinary metric expansiveness, and highlights the importance of the topology and dynamics near the boundary of the space.

A central theme emerging from these results is that metric-independent expansiveness is closely related to the possibility of realizing a non-compact system as the restriction of an expansive compact system whose boundary is dynamically isolated. This point of view not only clarifies the role of the one-point compactification, but also provides a flexible mechanism for producing new examples and sufficient conditions for metric-independent expansiveness.

Finally, we also introduce the notion of Cauchy expansiveness, formulated in terms of the large-scale behavior of Cauchy pairs and naturally adapted to uniform structures and completions. We prove that every uniformly continuous Cauchy expansive action extends uniquely to an expansive action on the completion. Although this perspective is somewhat different from the compactification approach, it fits naturally with the general philosophy of the paper, namely, that suitable control of the dynamics at infinity allows one to recover expansiveness from a genuinely topological viewpoint.

\section{Preliminaries}
Let \(X\) be a metrizable topological space and \(\varphi : G \times X \to X\) be a continuous action of the group $G$.
To simplify the notation we write $\varphi(g,x)=g \cdot x$.
If $d$ is a metric on $X$ we say that $\varphi$ is \textit{expansive with respect to} $d$ if
there exists \(c > 0\) such that for all \(x \neq y\) in \(X\), there exists \(g \in G\) such that
$$d(g \cdot x, g \cdot y) > c.$$

\begin{definition} \label{DefMet_Ind}
 We say that the action is \emph{metric-independent expansive}  (MIE) if it is expansive with respect to any metric compatible with the topology of the space.
%
%  for every metric \(d\) compatible with the topology of \(X\) there exists \(c > 0\) such that for all \(x \neq y\) in \(X\), there exists \(g \in G\) such that
% $$d(g \cdot x, g \cdot y) > c.$$
\end{definition}

Let us show that there are MIE homeomorphisms on non-compact spaces.

\begin{example}\label{ejemplo1}
Let $f : \Z \rightarrow \Z$ be defined by $f(x) = x + 1$. Clearly, if in $\Z$ we consider the usual metric then $f$ is expansive. 
But, in fact, $f$ is MIE.
Indeed, given any metric $d$ compatible with the topology, since $\{0\}$ is open, there exists $\epsilon > 0$ such that $B_d(0, \epsilon) \subset \{0\}$. Consequently, $\epsilon$ serves as an expansivity constant for $\varphi$ with respect to the metric $d$.
\end{example}

The next example is well-known and recalls that the expansivity may depend on the metric on a noncompact space.

\begin{example} \label{ejemplo2} 
Let $f : \R \rightarrow \R$ be the homothety with ratio $k > 1$, that is, $f(x) = kx$. It is clear that $f$ is an expansive homeomorphism with respect to the Euclidean metric. We define the metric 
\[
d'(x,y) = \dfrac{2\abs{x-y}}{\sqrt{(1 + x^2)(1 + y^2)}}
\]
given by the stereographic projection of $S^1$ onto $\R$. It is clear that $f$ is not an expansive homeomorphism with this metric.
\end{example}

\begin{remark}
%Let \( X \) be a metric space. 
MIE is preserved under conjugacy. It follows from the definition.
\end{remark}

It is easy to see that expansivity is preserved under restrictions to arbitrary invariant subsets (the same expansivity constant works for the restricted system).

\begin{proposition}\label{PreservMIE}
%Let \( X \) be a metric space. 
MIE is preserved under the restriction to a closed invariant subset.
% \( Y \subset X \).
\end{proposition}

\begin{proof}
%\emph{Conjugation}: It is straightforward to verify that \( \varphi_h : G \times Y \rightarrow Y \) is also an action MIE, where \( h : Y \rightarrow X \) is a homeomorphism.
%\emph{Restriction to closed invariant subsets}:
Let \( \varphi \) be a MIE action of $G$ on $X$ 
and let \( Y \subset X \) be a closed invariant subset. 
Denote as \( d_Y \) a compatible metric on $Y$. 
%Then the restricted action
%\[
%  G \curvearrowright Y
%\]
%is MIE.
By Hausdorff’s metric extension theorem \cites{Hausdorff1930,Torunczyk1972} there exists a (complete) metric \( D \) on \( X \) that coincides with \( d_Y \) on \( Y \). 
Since \( \varphi \) is MIE, 
it is expansive for $D$. Consequently, the restriction is expansive for $d_Y$ and the proof ends.
%%there exists \( \varepsilon > 0 \) such that for every distinct pair \( u, v \in X \), there exists \( g \in G \) such that
%\[
%  D(g \cdot u,\, g \cdot v) \;\ge\; \varepsilon.
%\]
%If \( y \ne y' \in Y \), then \( g \cdot y,\, g \cdot y' \in Y \) and
%\[
%  d_Y\bigl(g \cdot y,\, g \cdot y'\bigr)
%  = D\bigl(g \cdot y,\, g \cdot y'\bigr)
%  \;\ge\; \varepsilon.
%\]
%Since \( d_Y \) is the restriction of an arbitrary compatible metric on \( Y \), the restricted action is MIE.
\end{proof}

In Example \ref{exaRestNoMie} we show that the restriction
of a MIE homeomorphism to an invariant subset, which is not closed, may not be MIE.

\section{Cocompact expansivity}

In this section we assume that \(X\) is a
% locally compact
topological space and \(\varphi : G \times X \to X\) is an action.
We say that \(\varphi\) is \emph{cocompact} if there exists a compact subset  $K$ such that $G.K = X$.
% \begin{remark}
% Observe that in Example~\ref{ejemplo2}, the homeomorphism is not MIE, although it is cocompact.
% To characterize MIE actions on $\text{LC}\sigma$-spaces cocompactness alone is not sufficient, so, an additional condition must be imposed. This is precisely the purpose of the following definition.
% \end{remark}
For $A,B$ families of subsets of $X$ we write $A\prec B$ if
every $a\in A$ is contained in some $b\in B$.
For $a\subset X$ we write $a\prec B$ if $a$ is contained in some $b\in B$.
We say that the open cover $\U$ is an \textit{expansivity cover} if
$\{g\cdot x,g\cdot y\}\prec \U$
for all  $g\in G$ implies
$x=y$.
These definitions are used to build the next notion, which we will show that is strongly related to the MIE.

\begin{definition}\label{cocompactExP}
Let \(\mathcal{U}\)
% = \{ U_1, \dots, U_n \}\)
be an open cover of \(X\). We say that \(\varphi\) is \emph{cocompactly expansive}, with \(\mathcal{U}\) as an expansivity cover,
if there exists a compact set \(K \subset X\) such that:
\begin{enumerate}
\item \(G \cdot K = X\); that is, \(\varphi\) is cocompact with respect to \(K\).
\item \(\{ g \cdot x, g \cdot y \} \prec \mathcal{U} \cup (X \setminus K)\) for all \(g \in G\) implies \(x = y\).
\end{enumerate}
\end{definition}

If the property of cocompact expansivity is expressed purely in metric terms, the following equivalence follows immediately.

\begin{remark}
If \(X\) is metrizable then the following statements are equivalent:
\begin{enumerate}
\item \(\varphi\) is cocompactly expansive.
\item There exists a metric \(d\) compatible with the topology, a compact set \(K \subset X\) such that \(G \cdot K = X\) and a constant \(c > 0\) such that for all distinct \(x, y \in X\), there exists \(g \in G\) satisfying \(d(g \cdot x, g \cdot y) > c\) and \(\{ g \cdot x, g \cdot y \} \cap K \neq \emptyset\).
\end{enumerate}
That is, the points are separated when at least one of them is in the compact set $K$.
In particular, when $X$ is compact and metric we can take $K=X$ to see that expansivity is equivalent to cocompact expansivity.
\end{remark}

In what follows we explore more properties of cocompact expansivity.

\subsection{Preservation of Cocompact Expansivity under restrictions to subgroups}\label{preserv_op_MIE}

A subgroup \( H \le G \) is called \emph{syndetic} if there exists a finite set \( F \subset G \) such that \( G = F H \).

\begin{proposition}[Inheritance to syndetic subgroups] \label{th:herencia}
Let \( \varphi : G \times X \to X \) be a continuous and cocompactly expansive action and let \( H \) be a syndetic subgroup of \( G \). Then the restricted action \( \varphi|_{H \times X} : H \times X \to X \) is also continuous and cocompactly expansive.
\end{proposition}

\begin{proof}
Let \( K \) be the compact set from the cocompactness of \( \varphi \), \( \mathcal{U} = \{ U_1, \ldots, U_N \} \) be the expansivity cover, and \( F = \{ f_1, \ldots, f_m \} \subset G \) be the finite set such that \( G = F H \).
Each \( f_j^{-1}.K \) is compact, so \( K' := \bigcup_{j=1}^m f_j^{-1}.K \) is compact as well.
For any \( x \in X \), choose \( g \in G \), \( k \in K \) such that \( x = g.k \) (since \( G \cdot K = X \)).
Writing \( g = f_j h \) with \( f_j \in F \), \( h \in H \), we obtain
\[
x = f_j.(h.k) \in H \cdot K'.
\]
Hence, \( H \cdot K' = X \).
Next we define
\[
\mathcal{W}
:= \left\{
  \bigcap_{f \in F_0} f^{-1}.U_{i(f)} \;\middle|\;
  \varnothing \ne F_0 \subseteq F,\;
  i(f) \in \{1, \ldots, N\}
\right\}.
\]
Since \( \mathcal{U} \) and \( F \) are finite, the family \( \mathcal{W} \) is also finite.
Each element of \( \mathcal{W} \) is open (being a finite intersection of open sets) and contains all the sets \( f^{-1}.U_i \), so it covers \( X \).
Adding \( X \setminus K' \), we obtain a finite open cover:
$
\mathcal{W} \cup \{ X \setminus K' \}.
$
Let \( x \ne y \in X \). Since \( \varphi \) is cocompactly expansive, there exists \( g \in G \) such that:
\[
\{ g.x, g.y \} \cap K \ne \varnothing,
\quad \text{and} \quad
\forall U \in \mathcal{U},\;
\{ g.x, g.y \} \not\subset U.
\]
We write \( g = f_j h \) with \( f_j \in F \), \( h \in H \), and define \( x' := h.x \), \( y' := h.y \).
Then \( \{ g.x, g.y \} = f_j.\{ x', y' \} \), and one of these points lies in \( K \), so its preimage lies in \( f_j^{-1}.K \subset K' \).
Thus,
$
\{ x', y' \} \cap K' \ne \varnothing.
$
If \( \{ x', y' \} \subset X \setminus K' \), then applying \( f_j \) gives
\( \{ g.x, g.y \} \subset X \setminus K \), contradicting the assumption.

Now take \( W = \bigcap_{f \in F_0} f^{-1}.U_{i(f)} \in \mathcal{W} \) and fix \( f_k \in F_0 \). Then \( W \subset f_k^{-1}.U_{i(f_k)} \).
If \( \{ x', y' \} \subset W \), then applying \( f_k \) yields:
$
\{ f_k.x', f_k.y' \} \subset U_{i(f_k)} \subset \mathcal{U};
$
which contradicts the expansivity assumption.
Therefore, \( \{ x', y' \} \not\subset W \), and \( \varphi|_{H \times X} \) is cocompactly expansive.
\end{proof}

Notice that the subgroup \( n\mathbb{Z} \subset\Z \) is syndetic for all $n\neq 0$.

\begin{corollary}
\label{coroCocoExpPotencia}
 If \( f : X \to X \) is a cocompactly expansive homeomorphism then \( f^n \) is also cocompactly expansive for every \( n \in \mathbb{Z}, n\neq 0 \).
\end{corollary}

The following example shows that MIE and expansivity are not preserved under the restriction of the action to an arbitrary subgroup.

\begin{example}
Let \( G = \mathbb{Z}^2 \curvearrowright X = \{0,1\}^{\mathbb{Z}^2} \) be the shift action with the metric
\[
d(x, y) = \sum_{(i,j)\in\mathbb{Z}^2} 2^{-(|i|+|j|)} \bigl|x(i,j) - y(i,j)\bigr|.
\]
Although the action of \( G \) is metric-independent expansive (MIE), its restriction to the subgroup
\[
H = \mathbb{Z} \times \{0\} \subset G
\]
is not.
\end{example}
\begin{proof}
Let \( x_0 \) be the zero configuration. For each \( n \ge 1 \), define
\[
y_n(i,j) = \begin{cases}
1,& \text{if } (i,j) = (0,n),\\
0,& \text{otherwise.}
\end{cases}
\]

For \( h = (h_1, 0) \in H \), we have
\[
h \cdot x_0 = x_0, \quad h \cdot y_n(i,j) = y_n(i - h_1, j),
\]
and therefore,
\[
d(h \cdot x_0, h \cdot y_n) = 2^{-(|h_1| + n)} \le 2^{-n}.
\]
Given any \( \varepsilon > 0 \), choose \( N \) such that \( 2^{-N} < \varepsilon \). Then for this pair,
\[
d(h \cdot x_0, h \cdot y_N) \le 2^{-N} < \varepsilon \quad \forall\, h \in H,
\]
which shows that there is no positive expansive constant for \( H \). Therefore, the action \( H \curvearrowright X \) is not MIE.
\end{proof}

% \begin{remark}
% In Section~\ref{preserv_op_MIE}, we will see that if the space is a $\text{LC}\sigma$-space (see Definition \ref{def_sigma_compact}) and the subgroup is syndetic, MIE is preserved under passage to subgroups.
% \end{remark}

\subsection{Preservation of Cocompact Expansivity under Coproducts}

In this section we show the stability of cocompact expansivity under coproducts. The coproduct of an action on $X$ and an action on $Y$ is the induced action on the disjoint union $X\sqcup Y$ given by
the componentwise action
\[
g \cdot z :=
\begin{cases}
g.x & \text{if } z = x \in X, \\[2pt]
g.y & \text{if } z = y \in Y.
\end{cases}
\]

\begin{proposition}
\label{th:coproducto}
Let \( G \) be a group acting cocompactly expansive on two topological spaces \( X \) and \( Y \). Then the coproduct of the actions is also cocompactly expansive.
\end{proposition}

\begin{proof}
Assume the data \( (\mathcal{U}_X, K_X) \) and \( (\mathcal{U}_Y, K_Y) \) satisfy the definition of cocompact expansiveness for \( X \) and \( Y \), respectively.

\medskip
\noindent\textbf{1. Cocompact set.}
Let
\[
K := K_X \sqcup K_Y \subset X \sqcup Y.
\]
Since \( G \cdot K_X = X \) and \( G \cdot K_Y = Y \), it follows that \( G \cdot K = X \sqcup Y \).

\medskip
\noindent\textbf{2. Finite covering.}
Let
\[
\mathcal{U} := \mathcal{U}_X \cup \mathcal{U}_Y.
\]
This is a finite open covering of \( X \sqcup Y \), since \( \mathcal{U}_X \) covers \( X \) and \( \mathcal{U}_Y \) covers \( Y \).

\medskip
\noindent\textbf{3. Universal expansivity condition.}
Let \( z \ne w \in X \sqcup Y \). We consider two cases:

\emph{Case A: \( z, w \in X \) or \( z, w \in Y\).}
By cocompact expansiveness in \( X \), there exists \( g \in G \) such that
\[
\{g.z, g.w\} \cap K_X \ne \varnothing,
\quad
\forall U \in \mathcal{U}_X, \;
\{g.z, g.w\} \not\subset U.
\]
Since \( \{g.z, g.w\} \subset X \), this pair is not contained in any \( U \in \mathcal{U}_Y \), and we also have \( \{g.z, g.w\} \cap K = \{g.z, g.w\} \cap K_X \ne \varnothing \). So the pair satisfies the condition with respect to \( (\mathcal{U}, K) \) in \( X \sqcup Y \).

\emph{Case B: \( z \in X \), \( w \in Y \) or (vice versa).}
Let \( g \in G \) be such that \( g.z \in K_X \). Then \( \{g.z, g.w\} \cap K = \{g.z\} \ne \varnothing \). Since \( g.z \in X \) and \( g.w \in Y \), the pair cannot be contained in any \( U \in \mathcal{U}_X \) or \( U \in \mathcal{U}_Y \), because all such sets are contained either in \( X \) or in \( Y \), respectively. Hence, the universal expansivity condition is satisfied.

\medskip

In all cases, there exists \( g \in G \) such that the pair \( (z,w) \) satisfies the expansivity condition with respect to the data \( (\mathcal{U}, K) \). Therefore, the action on \( X \sqcup Y \) is cocompactly expansive.
\end{proof}

\begin{corollary}
If \( X \) and \( Y \) are $\text{LC}\sigma$ metric spaces, and if there exist actions \( \varphi : G \times X \to X \) and \( \phi : G \times Y \to Y \) that are metric-independent expansive (MIE), then the direct sum \( \varphi + \phi \) defines an MIE action on the coproduct \( X \sqcup Y \).
\end{corollary}

\begin{proof}
This follows directly from applying Proposition~\ref{th:coproducto} and Theorem~\ref{TeoEquiv}.
\end{proof}

\medskip

The following example shows that the property of being cocompactly expansive is not preserved under products.

\begin{example} \label{ex:no-producto}
Let \( G = \mathbb{Z} \), and consider the shift action
\[
  n \cdot x = x + n
  \qquad \text{on} \qquad
  X := Y := (\mathbb{Z}, \tau_{\mathrm{dis}}),
\]
with the discrete topology.

\begin{enumerate}
\item Each action \( G \curvearrowright X \) and \( G \curvearrowright Y \) is cocompactly expansive.
\item The diagonal action on \( X \times Y \),
  \[
  n \cdot (x, y) := (x + n,\; y + n),
  \]
  is not cocompact, and therefore not cocompactly expansive.
\end{enumerate}
\end{example}

\begin{proof}
We already saw that each individual action is metric-independent expansive. Since the space is $\text{LC}\sigma$-space, Theorem~\ref{TeoEquiv} implies that the actions are cocompactly expansive. However, the product action is not cocompact, and thus cannot be cocompactly expansive.

Indeed, the orbit of a point \( (x, y) \) under the diagonal action is
\[
\mathcal{O}(x, y) = \{(x + n,\; y + n) : n \in \mathbb{Z}\},
\]
which is an infinite straight line in \( \mathbb{Z}^2 \).

Let \( K \subset \mathbb{Z}^2 \) be compact (i.e., finite). The saturated set
\[
G \cdot K = \bigcup_{(x, y) \in K} \mathcal{O}(x, y)
\]
is a finite union of such lines.

However, covering \( \mathbb{Z}^2 \) requires \emph{infinitely many} distinct lines—one for each value of \( y - x \). Therefore, for any compact \( K \), we have \( G \cdot K \ne X \times Y \), and the diagonal action is not cocompact.
\end{proof}

\section{expansively extendible at a point}

Note that if a metric space is $\sigma$-compact but not locally compact, then its one-point compactification fails to be metrizable. Nevertheless, by employing the notion of cover expansivity, the following definition becomes entirely well-defined.

\begin{definition}
% Let \(X\) be a metrizable topological space, and let \(\varphi : G \times X \to X\) be an action.
We say that \(\varphi\) is \emph{expansively extendible at a point} if there exists an expansive action, in the sense of cover expansivity, \(\varphi' : G \times X' \to X'\), where \(X' = X \cup \{ p \}\) is the one-point compactification of \(X\), such that \(\varphi'|_{G \times X} = \varphi\) and \(g \cdot p = p\) for all \(g \in G\).
\end{definition}

Note that, within the class of metrizable locally compact spaces, being an \(\mathrm{LC}\sigma\)-space is precisely the condition ensuring that the Alexandroff compactification is metrizable; see \cite[Definition~5.1 and Theorem~5.3]{Kechris}.

\begin{theorem} \label{TeoEquiv}
Let \( X \) be a $\text{LC}\sigma$ metric space and let \( \varphi : G \times X \to X \) be an action. The following statements are equivalent:
\begin{enumerate}
\item \( \varphi \) is MIE.
\item \( \varphi \) is expansively extendible at a point.
\item \( \varphi \) is cocompactly expansive.
\end{enumerate} 
\end{theorem}

\begin{proof}
\textbf{\( (1) \Rightarrow (2) \):}  
Let \( \varphi' : G \times X' \to X' \) be the extension of \( \varphi \) to the Alexandroff compactification \( X' \) of \( X \). Suppose that there exists a metric \( \delta : X' \times X' \to \mathbb{R}^+ \) compatible with the topology of \( X' \) such that \( \varphi' \) is not expansive. Then, for every \( n \in \mathbb{N} \), there exist distinct points \( x_n, y_n \in X' \) such that for all \( g \in G \),
\[
\delta(g \cdot x_n, g \cdot y_n) \leq 2^{-n}.
\]

Clearly, if this inequality holds for some subsequence \( (x_{n_k}, y_{n_k}) \subset X \), it contradicts the expansiveness of \( \varphi \) with respect to the metric \( \delta|_{X \times X} \), and thus \( \varphi \) would not be MIE. Therefore, the remaining case to consider is when \( y_n = p \) for all \( n \in \mathbb{N} \). But then
\[
\delta|_{X \times X}(g \cdot x_n, g \cdot x_m) \leq \delta(g \cdot x_n, p) + \delta(p, g \cdot x_m) \leq 2^{-n} + 2^{-m},
\]
which also contradicts the expansiveness of \( \varphi \) with respect to the metric \( \delta|_{X \times X} \). 

Since the Alexandroff compactification is compact and metrizable, and since, as recalled in the introduction, expansiveness by coverings coincides with metric expansiveness in the compact metrizable case, the implication follows.

\vspace{0.5em}
\textbf{\( (2) \Rightarrow (1) \):}  
Let \( \varphi' : G \times X' \to X' \) be the one-point extension of \( \varphi \) that is expansive by coverings. Let \( \mathcal{U} = \{ U_1, \ldots, U_n \} \) be an expansivity cover, where \( U_n = X' \setminus K \) with \( K \subset X \) compact, and set \( \mathcal{V} = \{ U_1, \ldots, U_{n-1} \} \), a covering of \( K \).

Let \( d \) be a metric compatible with the topology of \( X \), and choose a compact \( K' \subset X \) with \( K \subset K' \) such that
\[
\text{dist}(K, X \setminus K') = \epsilon > 0.
\]
Define \( \mathcal{U}' = \{ U_1, \ldots, U_{n-1}, U_n' \} \), where \( U_n' := U_n \setminus \{ p \} \). Let \( \delta \) be the Lebesgue number of \( \mathcal{V} \) and \( \delta' \) that of \( \mathcal{U}' \) over \( K' \).

Let \( x, y \in X \) with \( x \ne y \). Then:

- If \( x, y \in K \), there exists \( g \in G \) such that \( \{ g \cdot x, g \cdot y \} \nprec \mathcal{U} \), so \( d(g \cdot x, g \cdot y) > \delta \).
- If \( x \in K \) and \( y \in U_n' \), then \( d(x,y) > \delta' \).
- If \( y \notin K' \), then \( y \in U_n \), and there exists \( g \in G \) such that \( \{ g \cdot y, g \cdot p \} \nprec \mathcal{U} \), so \( g \cdot y \in K \). If \( g \cdot x \in K' \), we fall into a previous case; otherwise \( d(g \cdot x, g \cdot y) > \epsilon \).

Taking \( c = \min \{ \epsilon, \delta, \delta' \} \), we conclude that \( c \) is an expansivity constant for \( \varphi \).

\vspace{0.5em}
\textbf{\( (2) \Rightarrow (3) \):}  
If \( \varphi \) is expansively extendible at a point, then by definition it is expansive. To satisfy the definition of cocompact expansiveness, it suffices to take as \( K \) the complement of the open set in the expansivity cover of \( \varphi \) that contains \( p \).

\vspace{0.5em}
\textbf{\( (3) \Rightarrow (2) \):}  
Let \( \{ U_1, \ldots, U_n \} \) and \( K \) be the expansivity cover and compact set given by the definition of cocompactly expansive action. Then
\[
\{ U_1, \ldots, U_n, X' \setminus K \}
\]
is an expansivity cover for \( \varphi' : G \times X' \to X' \).
\end{proof}

\begin{remark}
For homeomorphisms, the implication $(1)\Rightarrow(2)$ in the following theorem is already implicit in Bryant's work \cite{Bryant}.
\end{remark}

\begin{remark}
Observe that in Example \ref{ejemplo2}, the one-point compactification of \(\R\) is homeomorphic to \(S^1\)
and the circle does not admit expansive homeomorphisms. Therefore, by Theorem \ref{TeoEquiv} we conclude that the action is not MIE, without needing to construct an explicit metric.
\end{remark}

\begin{corollary}
 Moreover, if the space \( X \) is $\sigma$-compact
 and \( f \) is MIE, so is \( f^n \) for all \( n \in \mathbb{Z} \), with $n \neq 0$.
\end{corollary}

\begin{proof}
It follows from Theorem~\ref{TeoEquiv} and Corollary \ref{coroCocoExpPotencia}.
\end{proof}

\begin{remark}
Observe that in the proof of Theorem \ref{TeoEquiv}, we did not use the continuity of the action.
\end{remark}

% \begin{remark} \label{EquivContinua}
% Let \(X\) be a $\text{LC}\sigma$ metrizable space, and let \(\varphi : G \times X \to X\) be an action. The following statements are equivalent:
% \begin{enumerate}
%     \item \(\varphi\) is continuous and expansive independently of the metric.
%     \item \(\varphi\) is continuous and expansively extendible at a point.
%     \item \(\varphi\) is continuous and cocompactly expansive.
% \end{enumerate}
% \end{remark}

\begin{remark}
In \( (2 \Rightarrow 1) \), and \( (2  \Leftrightarrow 3) \) we used expansivity by coverings, which is equivalent to metric expansivity in the case where \(X'\) is metrizable. Therefore, we obtain the following result.
\end{remark}

\begin{theorem}
Let \(X\) be a topological space, and let \(\varphi : G \times X \rightarrow X\) be an action. Then \(\varphi\) is expansively extendible at a point if and only if \(\varphi\) is cocompactly expansive. Moreover, if either of these equivalent conditions holds, then \(\varphi\) is expansive independently of the metric.
\end{theorem}

\section{Metric-Independent Expansivity in Ordinal Spaces}

\begin{proposition}
Let \(X\) be a $\text{LC}\sigma$ metric space, and let \(\varphi : G \times X \rightarrow X\) be a continuous expansive action that is independent of the metric. If \(|G| \leq \aleph_0\), then the Alexandroff compactification \(X' = X \cup \{ p \}\) satisfies the second countability axiom.
\end{proposition}

\begin{proof}
By Theorem \ref{TeoEquiv}, \(\varphi\) can be extended to an expansive and continuous action on the one-point compactification \(X'\). Moreover, by Theorem 6.4 in \cite{Ferrari2025}), if \(\lvert G \rvert \leq \aleph_0\), then \(X'\) satisfies the second countability axiom.
\end{proof}

A consequence of the theorem by Kato and Park is the following theorem.

\begin{theorem}\label{th:derived-degree-not-limit}
Let \(X\) be a countable scattered metric space.  
Then \(X\) admits a homeomorphism \(f\colon X\to X\) that is MIE if and only if the derived degree of \(X\) is not an infinite limit ordinal.
\end{theorem}

\begin{proof}
If \(X\) is a countable scattered metric space, then its one-point compactification is also a countable scattered metric space. By Baker, \(X'\) is homeomorphic to an ordinal of the form \(\omega^{\alpha} + 1\) with the order topology. Since \(X\) admits an expansive homeomorphism independent of the metric by Theorem \ref{TeoEquiv}, \(X'\) also admits an expansive homeomorphism. Then, by Theorem 2.2 in \cite{Kato}, \( \deg(X') = \alpha \) is not an infinite limit ordinal. However, \(X\) is homeomorphic to \(\omega^{\alpha} n + 1\), and thus \(\deg(X') = \alpha\) or \(\alpha +1\). In any case, \(\deg(X)\) is not an infinite limit ordinal.

Conversely, since \(X'\) is homeomorphic to \(\omega^{\alpha} n + 1\), then \(X\) is homeomorphic to \(\omega^{\alpha} n\). If \(\alpha\) is an infinite limit ordinal, then \(\deg(X) = \alpha\); otherwise, \(\deg(X) = \alpha +1\). Thus, \(\alpha\) is not an infinite limit ordinal. Since \(\deg(X')= \alpha\), by Theorem 2.2 in \cite{Kato}, \(X'\) admits an expansive homeomorphism \(f\). However, in the construction of the expansive homeomorphism, \(x_{\infty}\) is the point of the one-point compactification and is fixed by the homeomorphism. Therefore, by Theorem  \ref{TeoEquiv}, the restriction \( f \upharpoonright X : X \to X \) is an expansive homeomorphism independent of the metric.
\end{proof}

\begin{theorem} 
A countable scattered metric space \(X\) admits a metric-independent \CBstable action if and only if the characteristic \((\alpha, n)\) of its one-point compactification satisfies either \(\alpha\) is not an infinite limit ordinal or \(n > 1\).
\end{theorem}

\begin{proof}
By Theorem 4.2 in \cite{Ferrari2025}, we know that \(X'\) admits an expansive \CBstable action if and only if the characteristic \((\alpha, n)\) satisfies either \(\alpha\) is not an infinite limit ordinal or \(n > 1\). However, the action can always be constructed with the point \(p\) of the compactification as a fixed point under the action.
\end{proof}

\section{Metric-Independent Expansivity in Totally Bounded Spaces}

\noindent

In this section we investigate MIE in totally bounded spaces, focusing in particular on those that arise as open subsets of a compact space.

\begin{definition}[Totally Bounded Space]
A metric space \(X\) is said to be \emph{totally bounded} if, for every \(\varepsilon > 0\), there exists a finite set \(\{x_1, x_2, \dots, x_n\} \subset X\) such that
\[
X \subseteq \bigcup_{i=1}^{n} B(x_i, \varepsilon),
\]
where \(B(x_i, \varepsilon) = \{y \in X : d(x_i, y) < \varepsilon\}\) is the open ball of radius \(\varepsilon\) centered at \(x_i\).
\end{definition}

\begin{definition}
Let \(X\) be a topological space and \(\varphi : G \times X \rightarrow X\) an action. We say that \(K \subset X\) is \emph{dynamically isolated} if there exists an open set \(U\) with \(K \subset U\) such that
\[
\bigcap_{g \in G} g.U = K.
\]
\end{definition}

Before proving the following theorem, we recall some fundamental facts about equivalence relations and prove a few lemmas.

\begin{definition}
Let \( X \) be a set and \( \sim \) an equivalence relation on \( X \). A subset \( A \subseteq X \) is said to be \emph{saturated} (with respect to \( \sim \)) if for every \( x \in A \), the equivalence class \( [x] \subseteq A \); that is, \( A \) contains entirely each equivalence class it intersects.
Equivalently, \( A \) is saturated if there exists a subset \( B \subseteq X/{\sim} \) such that \( A = \pi^{-1}(B) \), where \( \pi : X \to X/{\sim} \) is the canonical projection sending each point \( x \in X \) to its equivalence class \( [x] \in X/{\sim} \).
\end{definition}

\begin{lemma}\label{L:aislamiento-en-el-cociente}
Let \((X,d)\) be a compact metric space, \(\varphi\colon G\times X\to X\) a continuous action, and \(K\subset X\) a closed \(\varphi\)-invariant subset. Define the equivalence relation
\[
  x \sim y 
  \quad\Longleftrightarrow\quad 
  (\,x = y\,)\ \text{or}\ (\,x,y\in K\,),
\]
and let \(\pi\colon X\to Y := X/{\sim}\) be the canonical projection, with \(p = [K]\) denoting the equivalence class of all points in \(K\). Then:
\begin{enumerate}
  \item \(Y\) is compact.

  \item The induced action 
  \[
    \varphi'\colon G\times Y \;\longrightarrow\;Y,\qquad 
    \varphi'(g,[x]) := [\,\varphi(g,x)\,]
  \]
  is well defined and continuous.
  \item The restriction 
  \[
    \pi\bigl|_{\,X\setminus K}\;:\;X\setminus K \;\xrightarrow{\;\cong\;}Y\setminus\{p\}
  \]
  is a homeomorphism. In particular, \(X\setminus K\) and \(Y\setminus\{p\}\) are homeomorphic, and \(\pi\) maps \(K\) to the point \(p\).
\end{enumerate}
\end{lemma}

\begin{proof}
\begin{enumerate}
  \item Since \(Y=\pi(K)\) is the image of the compact set \(K\) under the continuous map \(\pi\), the quotient \(Y=X/K\) is compact.

  \item The induced action \(\varphi'\) is well defined because if \(x\in K\), then \(\varphi(g,x)\in K\) for all \(g\), and thus \([\varphi(g,x)] = p\). The continuity of \(\varphi'\) follows from the fact that \(\pi\) is open and saturated.
  \item The restriction \(\pi|_{\,X\setminus K}\) is continuous, bijective, and open onto \(Y\setminus\{p\}\). Its inverse maps \([x]\mapsto x\) for \(x\notin K\). This implies that it is a homeomorphism.
\end{enumerate}
\end{proof}

\begin{lemma}\label{L:aislamiento-dinamico}
Under the notation of Lemma \ref{L:aislamiento-en-el-cociente}, the subset \(K\subset X\) is dynamically isolated in \(X\) \emph{if and only if} \(\{p\}\subset Y\) is dynamically isolated for the action \(\varphi'\). That is,
\begin{align*}
&\exists\,U\subset X\ \text{open with }K\subset U \ \text{and}\\
&\qquad \bigcap_{g\in G} \varphi(g,U)=K\\[2pt]
&\quad\Longleftrightarrow\quad
\exists\,V\subset Y\ \text{open with }p\in V \ \text{and}\\
&\qquad \bigcap_{g\in G} \varphi'(g,V)=\{p\}.
\end{align*}

\end{lemma}

\begin{proof}
\begin{itemize}
  \item[\(\Rightarrow\)] Suppose that \(K\) is dynamically isolated in \(X\). Then there exists an open set \(U\subset X\) such that 
  \(K\subset U\) and \(\bigcap_{g}\varphi(g,U)=K\). Define \(V := \pi(U)\subset Y\). Since \(\pi\) is continuous and saturated, \(V\) is open and contains \(p\). For each \(y\neq p\) in \(Y\), choose \(x\notin K\) with \(\pi(x)=y\). Then \(x\notin \bigcap_{g}\varphi(g,U)\), so there exists \(g\) such that \(\varphi(g,x)\notin U\). Hence, \(\varphi'(g,y) = [\varphi(g,x)] \notin V\). Consequently, \(\bigcap_{g}\varphi'(g,V)=\{p\}\).
  
  \item[\(\Leftarrow\)] Suppose that \(\{p\}\) is dynamically isolated in \(Y\). Then there exists an open set \(V\subset Y\) with \(p\in V\) and 
  \(\bigcap_{g}\varphi'(g,V)=\{p\}\). Let \(U := \pi^{-1}(V)\subset X\). Then \(U\) is open and contains \(K = \pi^{-1}(\{p\})\). For any \(x\notin K\), there exists \(g\) such that \(\varphi'(g,[x])\notin V\), that is, \([\varphi(g,x)]\notin V\), which implies \(\varphi(g,x)\notin U\). Thus, we conclude that \(\bigcap_{g}\varphi(g,U)=K\).
\end{itemize}
\end{proof}

\begin{theorem}
\label{thmMIEIsolatedComplement}
Let \((X,d)\) be a compact metric space and let \(\varphi : G \times X \rightarrow X\) be an expansive action. Let \(K \subset X\) be a closed, \(\varphi\)-invariant subset. Then:
\[
  \varphi\bigl|_{\,X\setminus K}\text{ is MIE on }X \setminus K
  \quad\Longleftrightarrow\quad
  K\text{ is dynamically isolated.}
\]
\end{theorem}

\begin{proof}
By Proposition~\ref{PreservMIE} we know that
\[
\varphi\bigl|_{X \setminus K} \text{ is MIE on } X \setminus K 
\;\Longleftrightarrow\;
\varphi'\bigl|_{Y \setminus \{p\}} \text{ is MIE on } Y \setminus \{p\}.
\]
But \( X \setminus K \) is \(\sigma\)-compact, and therefore so is \( Y \setminus \{p\} \). Then, by Theorem~\ref{TeoEquiv},
\[
\varphi'\bigl|_{Y \setminus \{p\}} \text{ is MIE on } Y \setminus \{p\}
\;\Longleftrightarrow\;
\varphi' \text{ is MIE on } Y.
\]
Moreover, since \( Y \) is compact, this is equivalent to \(\varphi'\) being expansive on \( Y \), that is,
\[
\varphi' \text{ is expansive on } Y
\;\Longleftrightarrow\;
\{p\} \text{ is dynamically isolated in } Y.
\]
Finally, by Lemma~\ref{L:aislamiento-dinamico},
\[
\{p\} \text{ is dynamically isolated in } Y
\;\Longleftrightarrow\;
K \text{ is dynamically isolated in } X.
\]
This concludes the proof.
\end{proof}

\begin{remark}
\label{rmkMechMIE}
The previous theorem provides a method for constructing MIE actions as well as actions that are expansive but not MIE. It suffices to consider an expansive action on a compact space and a closed invariant subset \(K\). If \(K\) is dynamically isolated, then the restriction of the action to its complement is MIE.
If \(K\) is not dynamically isolated, then the restricted action is expansive with respect to the metric of $X$ restricted to $K$, but it is not MIE.
\end{remark}

The following examples illustrates the mechanism of Remark \ref{rmkMechMIE}.

\begin{example}
\label{Anosov}
One of the classical examples of an expansive homeomorphism (in fact, a diffeomorphism) is the Anosov automorphism on \(T^2\).
We denote 
\[
  T^2 = \R^2 / \Z^2
\]
the quotient space of \(\R^2\) by the lattice \(\Z^2\). Each point in \(T^2\) is denoted \([x]\), with \(x \in \R^2\), and \([x]\) its equivalence class modulo \(\Z^2\). We equip \(T^2\) with the distance
\[
  d([x],[y]) = \inf_{m \in \Z^2} \| x - y + m \|,
\]
where \(\| \cdot \|\) is the Euclidean norm on \(\R^2\).
Consider the hyperbolic matrix
\[
  A = \begin{pmatrix} 2 & 1 \\ 1 & 1 \end{pmatrix} \in SL(2,\Z).
\]
We define the linear diffeomorphism
\[
  f : T^2 \to T^2, 
  \qquad 
  f([x]) = [A x].
\]
Let \( K \) be the orbit of a periodic point. It is easy to verify that \( K \) is dynamically isolated. Then, by Theorem~\ref{TeoEquiv}, the Anosov map restricted to \( X \setminus K \) is MIE.
\end{example}

\begin{example}
\label{exaRestNoMie}
If $f\colon X\to X$ is expansive and has the shadowing property (as for example an Anosov diffeomorphism, a full shift or a subshift of finite type) then the dynamically isolated subsets are exactly the closed invariant subsets with the shadowing property. It is known that for a transitive, expansive homeomorphism with shadowing the periodic points are dense.
Thus, we can take an infinite minimal subset of $X$ in order to obtain a non-dynamically isolated subset $K\subset X$. As $K$ is closed its complement $X\setminus K$ is open.
Moreoever, if $X$ is connected (as if, for example, we are considering an Anosov diffeomorphism of a torus) $X\setminus K$ is not closed.
Therefore we can apply Theorem \ref{thmMIEIsolatedComplement}
to conclude that the restriction of $f$ to the non-closed subset $X\setminus K$ is not MIE. This complements Proposition \ref{PreservMIE}.
\end{example}

% We can construct another elementary example by defining a dynamical system on \( [0, \omega + 1] \).

The next example is another application of Theorem \ref{thmMIEIsolatedComplement}.

\begin{example}
Let \(X = [0, \omega ]\), let \(x_{\infty} := \omega \), and define \(f : X \rightarrow X\) as follows:
\[
f(x) = 
\begin{cases}
x + 2 & \text{if } x \text{ is even}, \\
x - 2 & \text{if } x \text{ is odd and greater than 1}, \\
0     & \text{if } x = 1, \\
x_{\infty} & \text{if } x = x_{\infty}.
\end{cases}
\]
It is easy to verify that \(f\) is an expansive homeomorphism and that \(K = \{x_{\infty}\}\) is a closed dynamically isolated set. Then, by the previous theorem, \(f \upharpoonright X \setminus K\) is metric-independent. We observe that \(f \upharpoonright X \setminus K\) is conjugate to the homeomorphism introduced in Example~\ref{ejemplo1}. This yields an alternative proof that the latter is a MIE homeomorphism.
\end{example}

\begin{proposition}
Let \(X\) be a compact scattered topological space, and let \(\varphi : G \times X \rightarrow X\) be a \CBstable action. Then there exists a closed, invariant, and proper dynamically isolated set \(K \subset X\).
\end{proposition}

\begin{proof}
We have already seen that $X$ is homeomorphic to $\omega^{\alpha}k + 1$ with the order topology. Therefore, we may assume, up to conjugation, that $X = \omega^{\alpha}k + 1$.

Let $Y = \lbrace \omega^{\alpha}i + 1 : i = 1, \ldots, k \rbrace$. Since $\varphi$ is \CBstable, it follows that $X$ is invariant under $\varphi$. 

If for every family of open-and-closed neighborhoods $U_i$ of $\omega^{\alpha}i + 1$, there exists $g \in G$ such that
\[
g\left( \bigcup_{i=1}^{k} U_i \right) \neq \bigcup_{i=1}^{k} U_i,
\]
then $K = Y$ is the closed dynamically isolated set. Otherwise, 
\[
K = \left( \bigcup_{i=1}^{k} U_i \right)^c
\]
is the closed dynamically isolated set.
\end{proof}

Using a similar reasoning as before, we obtain the following proposition. Note that in \cite{Artigue2024} we showed that in every countable compact space there exist doubly asymptotic points. The following result shows that in certain cases we can guarantee the existence of homoclinic points.

\begin{proposition}
Let $X$ be a countable compact metric space with characteristic $(\alpha + 1,1)$, and let $f : X \rightarrow X$ be an expansive homeomorphism. Then $f$ has nontrivial homoclinic points.
\end{proposition}

\begin{proof}
As we have already seen, we may assume that $X = \omega^{\alpha + 1} + 1$. Let 
\[
Y = \lbrace \omega^{\alpha}i + 1 : i \in \mathbb{N} \rbrace.
\]
Since $f$ is continuous, it is \CBstable; therefore, $Y$ is invariant under $f$. Observe that, due to the expansiveness of $f$, not all points can be fixed. Hence, there exists 
\[
x = \omega^{\alpha}j + 1, \quad \text{with } j \in \mathbb{N},
\]
such that
\[
\lim_{n \rightarrow +\infty} f^n(x) = \lim_{n \rightarrow -\infty} f^n(x) = \omega + 1.
\]
\end{proof}

\section{Metric-Independent Expansivity and Completion of Metric Spaces}

\begin{definition}
Let \(X\) be a metric space and \(\varphi : G \times X \rightarrow X\) an action. We say that it is uniformly continuous if the map \(T_g : X \rightarrow X\), defined by \(T_g(x) := g.x\), is uniformly continuous for all $g \in G$.
\end{definition}

\begin{definition}
Let \((X,d)\) be a metric space and \(\varphi : G \times X \rightarrow X\) a continuous action. We say that \(\varphi\) is \emph{Cauchy expansive} if there exists a constant \(c > 0\) such that for every sequence \((x_n, y_n)_{n \in \mathbb{N}} \subset X \times X\) satisfying:
\begin{enumerate}
    \item the sequence \((x_n, y_n)_{n \in \mathbb{N}} \subset X \times X\) is Cauchy in \(X\times X\),
    \item \(\inf_{n \in \mathbb{N}} d(x_n, y_n) > 0\),
\end{enumerate}
there exist \(n_0 \in \mathbb{N}\) and \(g \in G\) such that for all \(n \geq n_0\), one has
\[
d(\varphi(g,x_n), \varphi(g,y_n)) > c.
\]
\end{definition}

\begin{remark}
It is clear that Cauchy expansiveness implies expansiveness, simply by taking, for any \(x, y \in X\), the constant sequences \(x_n = x\) and \(y_n = y\).
\end{remark}

\begin{theorem}\label{ThmExtExp}
Let \((X,d)\) be a metric space and
\(\varphi : G \times X \longrightarrow X\) a Cauchy expansive action
and uniformly continuous. Let \(\hat{X}\) denote the completion of \(X\).
Then there exists a unique continuous and expansive action
\[
\hat{\varphi} : G \times \hat{X} \longrightarrow \hat{X},
\]
such that \(\hat{\varphi}\!\restriction_{\,X} = \varphi\);
moreover, the expansivity constant can be taken equal to that of \(\varphi\).
\end{theorem}

\begin{proof}
For each \(g \in G\), the map \(\varphi_g : X \to X\) is uniformly continuous,
and \(X\) is dense in \(\hat{X}\); the uniform extension theorem
provides a \emph{unique} continuous function
\(\hat{\varphi}_g : \hat{X} \to \hat{X}\)
with \(\hat{\varphi}_g\!\restriction_X = \varphi_g\).

\smallskip
For the identity element \(e\), we obtain
\(\hat{\varphi}_e = \operatorname{id}_{\hat{X}}\) by uniqueness.
If \(g,h \in G\), since \(\varphi_{gh} = \varphi_g \circ \varphi_h\) on \(X\),
the uniqueness of the extensions gives
\(\hat{\varphi}_{gh} = \hat{\varphi}_g \circ \hat{\varphi}_h\).
Defining
\(\hat{\varphi}(g,\hat{x}) := \hat{\varphi}_g(\hat{x})\),
we obtain a continuous action of \(G\) on \(\hat{X}\).

\smallskip
Let \(\varepsilon > 0\) be the expansivity constant of \(\varphi\).
Take \(\hat{x} \neq \hat{y}\) in \(\hat{X}\).
There exist sequences \(x_n, y_n \in X\) with \(x_n \to \hat{x}\) and \(y_n \to \hat{y}\).
Since \(\inf_n d(x_n, y_n) \ge \frac{1}{2} d(\hat{x}, \hat{y}) > 0\),
the pair \((x_n, y_n)\) satisfies the conditions to apply the hypothesis of Cauchy expansiveness, and thus there exist \(g \in G\) and \(n_0 \in \mathbb{N}\) such that
\(d(\varphi_g(x_n), \varphi_g(y_n)) > \varepsilon\) for all \(n \ge n_0\).
Passing to the limit and using the continuity of \(\hat{\varphi}_g\), we obtain:
\[
d\bigl(\hat{\varphi}_g(\hat{x}), \hat{\varphi}_g(\hat{y})\bigr)
= \lim_{n \to \infty} d\bigl(\varphi_g(x_n), \varphi_g(y_n)\bigr)
\;\ge\; \varepsilon.
\]
Therefore, \(\hat{\varphi}\) is expansive.
\end{proof}

\begin{corollary}
Let \(X\) be a totally bounded metric space, \(\varphi : G \times X \rightarrow X\) a uniformly continuous Cauchy expansive action, and suppose that \(\hat{X} \setminus X\) is closed.\\
If \(\hat{X} \setminus X\) is dynamically isolated, then \(\varphi\) is metric-independent.
\end{corollary}

\begin{proof}
Let \(\hat{\varphi} : G \times \hat{X} \rightarrow \hat{X}\) be the expansive extension of \(\varphi\). Then
\(\hat{\varphi} \upharpoonright (\hat{X} \setminus (\hat{X} \setminus X)) \equiv \varphi\). By the previous theorem, \(\varphi\) is metric-independent if \(\hat{X} \setminus X\) is dynamically isolated.
\end{proof}

\end{document}